\documentclass[11pt]{article}

\usepackage{amssymb, amsmath, amsthm, graphicx}
\usepackage[left=1in,top=1in,right=1in]{geometry}
\date{}

\theoremstyle{plain}
      \newtheorem{theorem}{Theorem}[section]
      \newtheorem{lemma}[theorem]{Lemma}

      \newtheorem{observation}[theorem]{Observation}
      
      \newtheorem{proposition}[theorem]{Proposition}
      \newtheorem{conjecture}[theorem]{Conjecture}
\theoremstyle{definition}

\theoremstyle{remark}

\def\conv{\mbox{\rm conv}}

\title{On the maximum number of $k$-holes in point sets with no $(k + 1)$-hole}

\begin{document}

\author{Andrew Suk\thanks{Department of Mathematics, University of California at San Diego, La Jolla, CA, 92093 USA. Supported by NSF grant DMS-2246847. Email: {\tt asuk@ucsd.edu}.}  \and Su Zhou\thanks{Department of Mathematics, University of California at San Diego, La Jolla, CA, 92093 USA. Supported by NSF grant DMS-2246847. Email: {\tt suzhou@ucsd.edu}.} }

\maketitle
\begin{abstract}
The classical problem of Erd\H{o}s asks for the minimum number of empty convex \(k\)-gons determined by an \(n\)-element point set in the plane. The celebrated empty hexagon theorem, proved independently by Gerken and Nicol\'as, shows that every sufficiently large planar point set contains a 6-hole, while Horton’s famous construction shows the existence of arbitrarily large point sets with no 7-hole. In this paper, we initiate the study of the maximum number of \(k\)-holes in planar point sets with no \((k+1)\)-hole. More precisely, for each fixed \(k\ge 6\), let \(h_k(n)\) be the maximum number of \(k\)-holes determined by a planar point set in general position, of size at most \(n\), and with no \((k+1)\)-hole. We prove that there are absolute constants $c_1,c_2>0$ such that
\[
\left(\frac{c_1}{k}\right)^{\lfloor k/3\rfloor} n^{\lfloor k/3\rfloor}
\le h_k(n)\le
\left(\frac{c_2}{k}\right)^{\lceil k/2\rceil} n^{\lceil k/2\rceil}.
\]
\end{abstract}
 
\section{Introduction}

A finite point set in the plane is in \emph{convex position} if it is the vertex set of a convex polygon, and is in \emph{general position} if no three points are collinear. Let \(ES(k)\) denote the minimum integer such that every planar point set in general position with at least \(ES(k)\) points contains \(k\) points in convex position. Classical theorems of Erd\H{o}s and Szekeres \cite{ES35,es2} show that
\[
2^{k-2}+1\le ES(k)\le \binom{2k-4}{k-2}+1,
\]
and they famously conjectured that the lower bound is sharp. After a long sequence of improvements \cite{CG,KP,TV1,TV2}, the first author \cite{S} proved that this conjecture is asymptotically correct by establishing the bound \(ES(k)=2^{k+o(k)}\). The current best bound on the lower-order term is due to Holmsen, Mojarrad, Pach, and Tardos \cite{holm}, who showed that $
ES(k)\le 2^{k+O(\sqrt{k\log k})}.$

A variant of the Erd\H{o}s--Szekeres problem concerns empty convex polygons. Given a finite point set \(P\), a \(k\)-element subset \(X\subset P\) is called a \emph{\(k\)-hole} if the points of \(X\) are in convex position and the interior of \(\conv(X)\) contains no point of \(P\). A celebrated construction of Horton~\cite{H} shows that arbitrarily large planar point sets in general position may contain no \(7\)-hole, while the empty hexagon theorem of Gerken~\cite{Gerken} and independently Nicol\'as~\cite{Nicolas} shows that every sufficiently large planar point set in general position contains a \(6\)-hole.

For each fixed integer $k \ge 3$, let $f_k(n)$ denote the minimum number of
$k$-holes determined by an $n$-element point set in the plane in general
position. Hence, Horton's construction implies that $f_k(n)=0$ for $k \ge 7$.
For $k \le 6$, estimating the asymptotic behavior of $f_k(n)$ has attracted
considerable attention. It is known that $f_3(n)=\Theta(n^2)$ and
$f_4(n)=\Theta(n^2)$, see, e.g., \cite{BV,Ai} for the best known constants.
The recursive construction of Horton determines only \(\Theta(n^2)\)
\(5\)-holes and \(\Theta(n^2)\) \(6\)-holes (see~\cite{BV}). It is conjectured that these
quadratic upper bounds are tight, that is,
\[
        f_5(n)=\Theta(n^2)
        \qquad\text{and}\qquad
        f_6(n)=\Theta(n^2).
\]
For many years, the best known lower bound for \(f_5(n)\) was linear.
Aichholzer et al.~\cite{AiB} established the first superlinear bound, proving
that \(f_5(n)=\Omega(n(\log n)^{4/5})\). Very recently,
Astudillo-Marb\'an and Sol\'e-Pi~\cite{AMS} announced a proof of the improved
bound \(f_5(n)=\Omega(n^{20/11})\). For \(k=6\), no superlinear lower bound is known.  The empty hexagon theorem
of Gerken~\cite{Gerken} and Nicol\'as~\cite{Nicolas} implies only a linear
lower bound.

Without any additional restriction, the maximum number of \(k\)-holes determined
by a point set of size at most \(n\) is clearly \(\binom{n}{k}\), since every
\(k\)-subset of the vertices of a convex \(n\)-gon forms a \(k\)-hole.  We
consider the following Tur\'an-type problem: how many \(k\)-holes can a point
set of size at most \(n\) determine if it contains no \((k+1)\)-hole?  

For each fixed integer \(k\ge 6\), let \(h_k(n)\) be the maximum number of
\(k\)-holes determined by a planar point set in general position, of size at
most \(n\), and containing no \((k+1)\)-hole.  In other words, any point set of size at most \(n\) with more than \(h_k(n)\)
\(k\)-holes guarantees the existence of a \((k+1)\)-hole. Note that $h_k(n)$ is clearly monotone by definition. We prove the following.

\begin{theorem}\label{hole1}
There are absolute constants $c_1,c_2>0$ such that the following holds. For every integer \(n>k\ge 6\),
\[
\left(\frac{c_1}{k}\right)^{\lfloor k/3\rfloor} n^{\lfloor k/3\rfloor}
\le h_k(n)\le
\left(\frac{c_2}{k}\right)^{\lceil k/2\rceil} n^{\lceil k/2\rceil}.
\]
\end{theorem}

\medskip

We believe that the correct exponent of $n$ is closer to the lower bound in
Theorem~\ref{hole1} than to the upper bound.  This is already consistent with
the behaviors of the Horton sets and the grid Horton sets (see~\cite{BV}), which both contain no \(7\)-hole but determine
\(\Theta(n^2)\) \(k\)-holes for each \(k=3,4,5,6\).  We conjecture the
following.

\begin{conjecture}
For every fixed integer $k\ge 6$, we have
\[
h_k(n)=O_k\!\left(n^{\lceil k/3 \rceil}\right).
\]
\end{conjecture}

\section{Horton sets and preliminary lemmas}

We begin by recalling Horton's classical construction, which requires the following definitions.
Given a finite point set $P$ in the plane, we say that $P$ is in \emph{strongly general position} if $P$ is in general position and no two points in $P$ share the same $x$-coordinate.  Let $X$ be an $r$-element point set in the plane in general position. We say that $X$
forms an \emph{$r$-cup} (\emph{$r$-cap}) if $X$ is in convex position and its convex hull is bounded above (below) by a single edge.

Given an $r$-cup $X \subset P$, we say that $X$ is \emph{open} if no point of $P$ lies above the $r$-cup, that is, no point of $P$ lies strictly above some point of the lower chain of $\conv(X)$. Likewise, given an $r$-cap $X \subset P$, we say that $X$ is \emph{open} if no point of $P$ lies below the $r$-cap, that is, no point of $P$ lies below the upper chain of $\conv(X)$. We say that a finite point set $P$ is \emph{upper $r$-closed} (\emph{lower $r$-closed}) if it contains no open $r$-cups (respectively, no open $r$-caps).

Given two finite point sets $P$ and $Q$ in the plane, we say that $P$ lies \emph{high above} $Q$ (or $Q$ lies \emph{deep below} $P$) if $P$ lies above every line generated by two points in $Q$ and $Q$ lies below every line generated by two points in $P$.

Let $H=\{h_1,\ldots,h_n\}$ be an $n$-element point set in the plane in strongly general position, ordered with respect to increasing $x$-coordinates. We say that $H$ is a \emph{Horton set} if it can be obtained by finitely many applications of the following rules:
\begin{enumerate}
    \item Every empty set and every one-point set is a Horton set.
    \item If the sets $H^{(1)}=\{h_i\in H:\textnormal{\emph{i} is odd}\}$ and $H^{(2)}=\{h_i\in H:\textnormal{\emph{i} is even}\}$ are Horton sets, and $H^{(1)}$ lies deep below or high above $H^{(2)}$, then $H$ is a Horton set.
\end{enumerate}

\noindent In \cite{H}, Horton constructed the following Horton set $P_t$ of size $2^t$ that contains no 7-hole:
\[
P_t=\{(i,d_i):0\le i<2^t\},
\]
where
\[
d_i=\sum_{j=1}^t a_j(2^t+1)^{t-j},
\]
and $i=\sum_{j=1}^t a_j2^{j-1}$ is the binary expansion of $i$. Since the recursive definition of a Horton set is preserved under taking subsequences in increasing \(x\)-coordinate order, every such subsequence is again a Horton set.

\begin{observation}
For each integer $m\ge1$, there is a Horton set of size $m$ in the plane.
\end{observation}

We will also need the following basic lemmas concerning Horton sets.

\begin{lemma}[\cite{H}]
Every Horton set is both upper 4-closed and lower 4-closed, and contains no $7$-hole. 
\end{lemma}

\begin{lemma}\label{lem:3cap}
Every Horton set of size $n$ has at least $\left\lfloor \frac{n-2}{2}\right\rfloor$ open 3-caps, and at least $\left\lfloor \frac{n-2}{2}\right\rfloor$ open 3-cups.
\end{lemma}

\begin{proof}
    If $n \leq 2$, then the statement is trivial.  Assume that $n\geq 3$, and let $H  = \{h_1,\ldots, h_n\}$.  Set $H^{(1)} = \{h_i \in H: \textnormal{\emph{i} is odd}\}$ and $H^{(2)} = \{h_i \in H: \textnormal{\emph{i} is even}\}$.  Without loss of generality, we can assume that $H^{(1)}$ lies deep below $H^{(2)}$, since a symmetric argument would follow otherwise. For each odd integer \(i\le n-2\), the triple
\((h_i,h_{i+1},h_{i+2})\)
forms an open 3-cap, since \(H^{(1)}\) lies deep below \(H^{(2)}\).
Similarly, for each even integer \(i\le n-2\), the triple \((h_i,h_{i+1},h_{i+2})\) forms an open 3-cup.
This gives at least
\(\lfloor (n-2)/2\rfloor\)
open 3-caps and the same number of open 3-cups.
\end{proof}

We say that $H$ is a \emph{Horton set with respect to a segment $s$} if, after rotating the plane so that $s$ becomes horizontal, the image of $H$ is a Horton set.  Likewise, we say that $X \subset H$ is an \emph{open cap (cup) with respect to segment $s$} if, after properly rotating the plane so that $s$ becomes horizontal, the image of $X$ is an open cap (cup) in $H$. We are now ready to prove Theorem \ref{hole1}.

\section{Upper and lower bounds for $h_k(n)$}
\begin{proof}[Proof of Theorem \ref{hole1}] We start by proving the lower bound for $h_k(n)$.   Set \(m=\lfloor k/3\rfloor\), and consider the convex polygon \(C\) with \(m+1\) vertices forming a cap. More specifically, let the vertices of \(C\) be
\[
\left(\cos\left(\frac{\pi i}{m}\right),\sin\left(\frac{\pi i}{m}\right)\right)\in\mathbb{R}^2,
\]
for \(i=0,\ldots,m\). Thus, the upper boundary of \(C\) consists of \(m\) sides \(s_1,\ldots,s_m\), ordered from left to right.  Along each side \(s_i\), near its midpoint, we place a sufficiently small and flat copy \(H_i\) of a Horton set as follows.  Since there exist integers \(q\ge 0\) and \(0\le r<m\) such that
\[
n-(k\bmod 3)=q m+r
\]
where $(k\bmod 3)$ denotes the remainder of $k$ when divided by $3$, we let \(H_1,\dots,H_r\) be Horton sets of size \(q+1\), and \(H_{r+1},\dots,H_{m}\) be Horton sets of size \(q\). Then
\[
\sum_{i=1}^{m}|H_i|
=
r(q+1)+\left(m-r\right)q
=
qm+r
=
n-(k\bmod 3).
\]
Each such $H_i$ will be chosen so that \(H_i\) is a Horton set with respect to \(s_i\). Moreover, the points of each \(H_i\) are ordered by increasing \(x\)-coordinate, and this ordering is preserved after rotating the plane so that \(s_i\) becomes horizontal.

We set $P =  H_1\cup \cdots \cup H_m$, which lies above the $x$-axis. Let $P_0$ be a set of $k\bmod 3$ points placed on the $x$-axis very close to the origin and set $P^{\ast} = P_0\cup P$. By applying small perturbations, we may assume that $P^{\ast}$ is in general position. Hence $|P^{\ast}|=n$, and by choosing each copy of $H_i$ sufficiently small and flat, the set $P^{\ast}$ has the following properties.

\begin{enumerate}
    \item If $\ell$ is the line generated by any two points in $H_i$, then $P^{\ast}\setminus H_i$ lies below $\ell$.

\item For each \(H_i=\{h_1,\ldots,h_{|H_i|}\}\), every line \(\ell\) passing through \(h_j\) and any other point of \(P^{\ast}\setminus H_i\) has the property that the points \(\{h_1,\ldots,h_{j-1}\}\) lie on one side of \(\ell\), while the points \(\{h_{j+1},\ldots,h_{|H_i|}\}\) lie on the other side.
\end{enumerate}

    \begin{figure}
        \centering
        \includegraphics[width=0.4\linewidth]{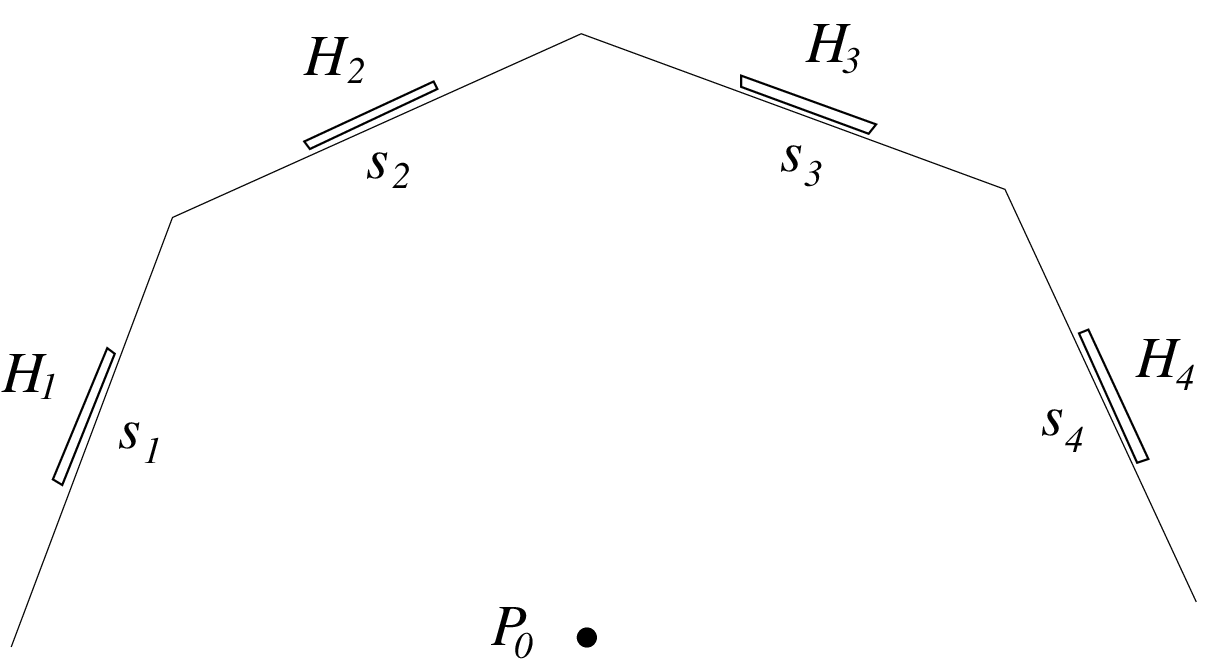}
        \caption{Example with $k = 13$, $m = 4$ and $|P_0| = 1$.}
        \label{fig:const}
    \end{figure}

\noindent See Figure \ref{fig:const}.   We now have the following simple observation.

\begin{observation}\label{clean}
    Let $X \subset H_i$ be a cap with respect to segment $s_i$ and $p \in P^{\ast}\setminus H_i$.  Then $X\cup p$ is a hole in $P^{\ast}$ if and only if $X$ is an open cap with respect to $s_i$.  
\end{observation}

\begin{proof}
Let $X=\{h_{i_1},\ldots,h_{i_r}\}\subset H_i$ be ordered with respect to their \(x\)-coordinate and let $p \in P^{\ast}\setminus H_i$.  Since each $H_j$ is sufficiently small,  every point of \(P^{\ast}\setminus (H_i\cup \{p\})\) lies outside \(\conv(X\cup\{p\})\), so any point of \(P^{\ast}\) contained in the interior of \(\conv(X\cup\{p\})\) must belong to \(H_i\setminus X\).

Suppose first that \(X\) is an open cap with respect to \(s_i\). Then no point of \(H_i\setminus X\) lies in \(\conv(X\cup\{p\})\). Indeed, let \(h_t\in H_i\setminus X\). If \(t<i_1\) or \(t>i_r\), then property~(2) implies that \(h_t\) does not lie between the lines generated by \(ph_{i_1}\) and \(ph_{i_r}\).  Hence, \(h_t\notin \conv(X\cup\{p\})\). If \(i_1<t<i_r\), then, since \(X\) is an open cap, the point \(h_t\) lies above the polygonal chain \(h_{i_1}h_{i_2}\cdots h_{i_r}\). Thus again \(h_t\notin \conv(X\cup\{p\})\). It follows that \(\conv(X\cup\{p\})\) contains no point of \(P^{\ast}\setminus (X\cup\{p\})\), and hence \(X\cup\{p\}\) is a hole in \(P^{\ast}\).

Conversely, suppose that \(X\cup\{p\}\) is a hole in \(P^{\ast}\). If \(X\) were not an open cap, then there would exist a point
\(h_t\in H_i\setminus X\) with \(i_1<t<i_r\) lying below the polygonal chain
\(h_{i_1}h_{i_2}\cdots h_{i_r}\).
Since \(i_1<t<i_r\), property~(2) implies that \(h_t\) lies between the two
lines \(ph_{i_1}\) and \(ph_{i_r}\).
Hence \(h_t\in \conv(X\cup\{p\})\), a contradiction.\end{proof}

By Observation \ref{clean}, we can select an open 3-cap $X_i$ in each $H_i$, and together with $P_0$, to create a $k$-hole in $P^{\ast}$.  Indeed, by property (1), (2), $X_1\cup \cdots \cup X_m$ forms a cap, and together with $P_0$, we have $k$ points in convex position.  By Observation \ref{clean} and property (2), $X_1\cup \cdots \cup X_m \cup P_0$ is a hole.  Since each $H_i$ has $\Omega(n/m)$ open 3-caps by Lemma \ref{lem:3cap}, we obtain at least $\Omega((n/m)^{\lfloor k/3\rfloor})$ holes of size $k$ in $P^{\ast}$.  Since $ m = \lfloor k/3\rfloor$, there is an absolute constant $c_1 > 0$ such that   

$$h_k(n) \geq \left(\frac{c_1}{k}\right)^{\lfloor k/3\rfloor }n^{\lfloor k/3\rfloor}.$$

Moreover, \(P^{\ast}\) does not contain a \((k+1)\)-hole.  Indeed, suppose
for contradiction that \(X\subset P^{\ast}\) is a \((k+1)\)-hole.  Since each
\(H_i\) contains no \(7\)-hole, the set \(X\) is not contained in any single
\(H_i\).  Moreover, if \(|X\cap H_i|\le 3\) for every \(i\), then
\[
        |X|\le 3m+|P_0|=k,
\]
a contradiction.  Hence, for some \(i\), we have \(|X\cap H_i|\ge 4\) and
\(X\cap(P^{\ast}\setminus H_i)\neq\emptyset\).  Thus \(X\cap H_i\) must form
a cap with respect to \(s_i\).  Since \(H_i\) is lower \(4\)-closed with
respect to \(s_i\), Observation~\ref{clean} implies that \(X\) is not a
\((k+1)\)-hole, a contradiction.

\medskip

Next, we prove the upper bound for \(h_k(n)\). Choose \(c_2>0\) sufficiently large so that \((c_2/2e)^3>2\).  For the sake of contradiction, suppose \(P\) is a point set of size \(n'\le n\) in the plane with no \((k+1)\)-hole, but with at least \(\left(c_2/k\right)^{\lceil k/2\rceil} n^{\lceil k/2\rceil}\) \(k\)-holes, where \(c_2\) is a sufficiently large constant. Since \(n'\le n\), the set \(P\) has at least \(\left(c_2/k\right)^{\lceil k/2\rceil}(n')^{\lceil k/2\rceil}\) \(k\)-holes. After rotation, we may assume that no two points from $P$ share the same $x$-coordinate. 

Let $X = \{x_1,\ldots, x_k\}\subset P$ be a convex $k$-gon in $P$ such that $x_1$ is the left-most point in $X$, and the elements in $X$ are ordered in clockwise order.  We say that $Y\subset X$ is a \emph{support} of $X$ if $Y  = \{x_1,x_3,\ldots, x_{k-1}\}$ when $k$ is even, and $Y = \{x_1,x_3,\ldots, x_k\}$ when $k$ is odd. Hence, each $k$-hole $X$ in $P$ corresponds to a unique support $Y\subset X$.  Since $|Y| = \lceil k/2\rceil$, and since \(P\) contains at least \(\left(c_2/k\right)^{\lceil k/2\rceil}(n')^{\lceil k/2\rceil}\) \(k\)-holes, there must be a subset $Y\subset P$ of size $\lceil k/2\rceil$ such that $Y$ is the support for at least

\begin{align*}
\frac{\left(c_2/k\right)^{\lceil k/2\rceil} (n')^{\lceil k/2\rceil}}
{\binom{n'}{\lceil k/2\rceil}}
&\ge
\left(\frac{c_2}{k}\right)^{\lceil k/2\rceil}\lceil k/2\rceil!\, 
\\
&>
\left(\frac{c_2}{k}\right)^{\lceil k/2\rceil}
\left(\frac{\lceil k/2\rceil}{e}\right)^{\lceil k/2\rceil}  \\
&=
\left(\frac{c_2\lceil k/2\rceil}{ek}\right)^{\lceil k/2\rceil}  \\
&\ge
\left(\frac{c_2}{2e}\right)^{\lceil k/2\rceil} 
>2,
\end{align*}

\noindent $k$-holes in $P$.  
Let $X,X' \subset P$ be two distinct $k$-holes in $P$ that use $Y$ as a support.  Hence, we have $X = \{x_1,\ldots, x_k\}$, $X' = \{x'_1,\ldots, x'_k\}$, $x_1 = x'_1$, $x_3 = x'_3$, etc. Just as above, $x_1$ is the left-most point in $X$, and the elements in $X$ are labeled in clockwise order.  Likewise, \(x'_1=x_1\) is the left-most point of \(X'\), and the elements of \(X'\) are labeled in clockwise order. We are going to finish the proof using the following observation.

\begin{observation}\label{diff}

Each $Y\subseteq P$ with $|Y|=\lceil k/2\rceil$ supports at most one $k$-hole $X\subseteq P$.  
    
\end{observation}

\begin{proof}

    \begin{figure}
        \centering
        \includegraphics[width=0.8\linewidth]{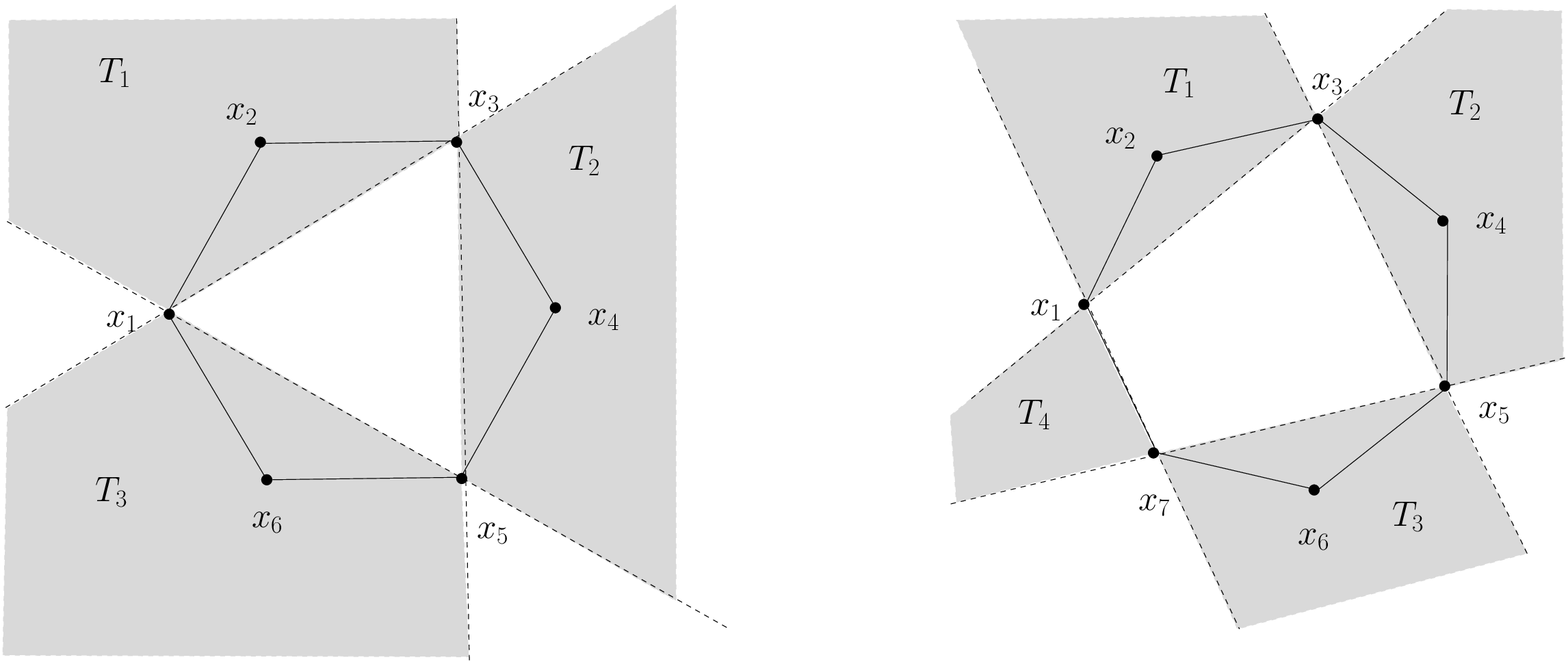}
        \caption{Regions $T_i$ shaded in gray for $k = 6$, and $k = 7$.}
        \label{exconv}
    \end{figure}

Suppose $Y$ supports two distinct $k$-holes $X$ and $X'$ as above. We define regions \[T_1,\ldots,T_{\lfloor (k+1)/2\rfloor}\] outside of \(\conv(Y)\) as follows. If \(k\) is even, then for each \(1\le i\le k/2\), let \(T_i\) be the region outside of \(\conv(Y)\) bounded by the segment \(x_{2i-1}x_{2i+1}\) and by the two lines \(x_{2i-3}x_{2i-1}\) and \(x_{2i+1}x_{2i+3}\), where indices are taken modulo \(k\).  Now suppose that \(k\) is odd. For each \(1\le i\le (k-1)/2\), let \(T_i\) be the region outside of \(\conv(Y)\) bounded by the segment \(x_{2i-1}x_{2i+1}\) and by the two lines containing the adjacent sides of \(\conv(Y)\), and let \(T_{(k+1)/2}\) be the region outside of \(\conv(Y)\) bounded by the segment \(x_1x_k\) and by the two lines \(x_{k-2}x_k\) and \(x_1x_3\).  See Figure \ref{exconv}.

  First, let us consider the case when $k$ is even.  Suppose $X,X'$ share some point not in $Y$.  Let $i\leq k/2$ be the maximum integer such that $x_{2i} \neq x'_{2i}$ and $x_{2i + 2} = x'_{2i+2}$, where $x_{k + 1} = x_1$ and $x_{k + 2} = x_2$.  We define $S_i$ to be the region inside of $T_i$, bounded by the lines $x_{2i-3}x_{2i-1}$, $x_{2i-2}x_{2i-1}$, and $x_{2i}x_{2i+1}$.   Likewise, we define $R_i$ to be the region inside of $T_i$, bounded by the lines $x_{2i-1}x_{2i}$, $x_{2i+1}x_{2i+2}$, and $x_{2i + 1}x_{2i+3}$.  See Figure \ref{fig:es1}.

We start in region $T_i$.  Since both $X$ and $X'$ are $k$-holes, and since $P$ does not contain a $(k + 1)$-hole, this implies that $x'_{2i}$ must lie in $S_i$ or $R_i$: otherwise, we must have that $X\cup \{x'_{2i}\}$ is a $(k+1)$-hole, or $\conv(\{x'_{2i - 1},x'_{2i},x'_{2i + 2}\})$/$\conv(\{x_{2i - 1},x_{2i},x_{2i + 2}\})$ is not empty.  See Figure \ref{fig:es1}.  If $x'_{2i} \in R_i$, then $x'_{2i + 1} \in \conv(\{x'_{2i - 1},x'_{2i},x'_{2i + 2}\})$ since $x_{2i + 2} = x'_{2i+2}$.  This contradicts the fact that $X'$ is a $k$-hole.  Therefore, we must have $x'_{2i} \in S_i$ and we proceed to region $T_{i-1}$.  

We apply the same argument inside region $T_{i-1}$.  Again, we must have $x'_{2i - 2} \neq x_{2i - 2}$ and $x'_{2i - 2}  \notin R_{i-1}$, since otherwise $x'_{2i - 1} \in \conv(\{x'_{2i - 3},x'_{2i-2},x'_{2i}\})$ and $X'$ would not be a $k$-hole.  Hence, we must have $x'_{2i-2} \in S_{i-1}$ and we proceed to region $T_{i-2}$.  Iterating this argument in counterclockwise order over all indices forces $x_{2i+2} \neq x'_{2i+2}$, contradicting the choice of $i$. Hence, we must have that $x'_{2i}\in S_i$ or $x'_{2i}\in R_i$ for every $i$. However, then $x'_2$ or $x'_{k}$ will be the left-most vertex of $X'$, which contradicts the fact that $x_1$ is the left-most vertex of $X'$. 

\medskip

Now suppose that $k$ is odd, which implies that the interior of region $T_{\lceil k/2\rceil}$ does not contain a point from $X$ and $X'$. Since both $X$ and $X'$ are $k$-holes, and since $P$ does not contain a $(k + 1)$-hole, this implies that $x_{k-1}=x'_{k-1}$ or $x'_{k-1} \in S_{\lceil k/2\rceil-1}$. As $X\neq X'$, let $i\leq (k-1)/2$ be maximal such that $x_{2i} \neq x'_{2i}$.

We start in region $T_i$, and argue just as above.  Since both $X$ and $X'$ are $k$-holes, and since $P$ does not contain a $(k + 1)$-hole, this implies that $x'_{2i}$ must lie in $S_i$ or $R_i$.  However, since we are assuming $i$ is maximal such that $x_{2i} \neq x'_{2i}$, this implies that $x'_{2i} \notin R_i$ since otherwise $x'_{2i + 1} \in \conv(\{x'_{2i - 1},x'_{2i},x'_{2i + 2}\})$ and $X'$ would not be a $k$-hole.  Therefore, $x'_{2i} \in S_i$ and we proceed to region $T_{i - 1}$.

    \begin{figure}
        \centering
        \includegraphics[width=0.7\linewidth]{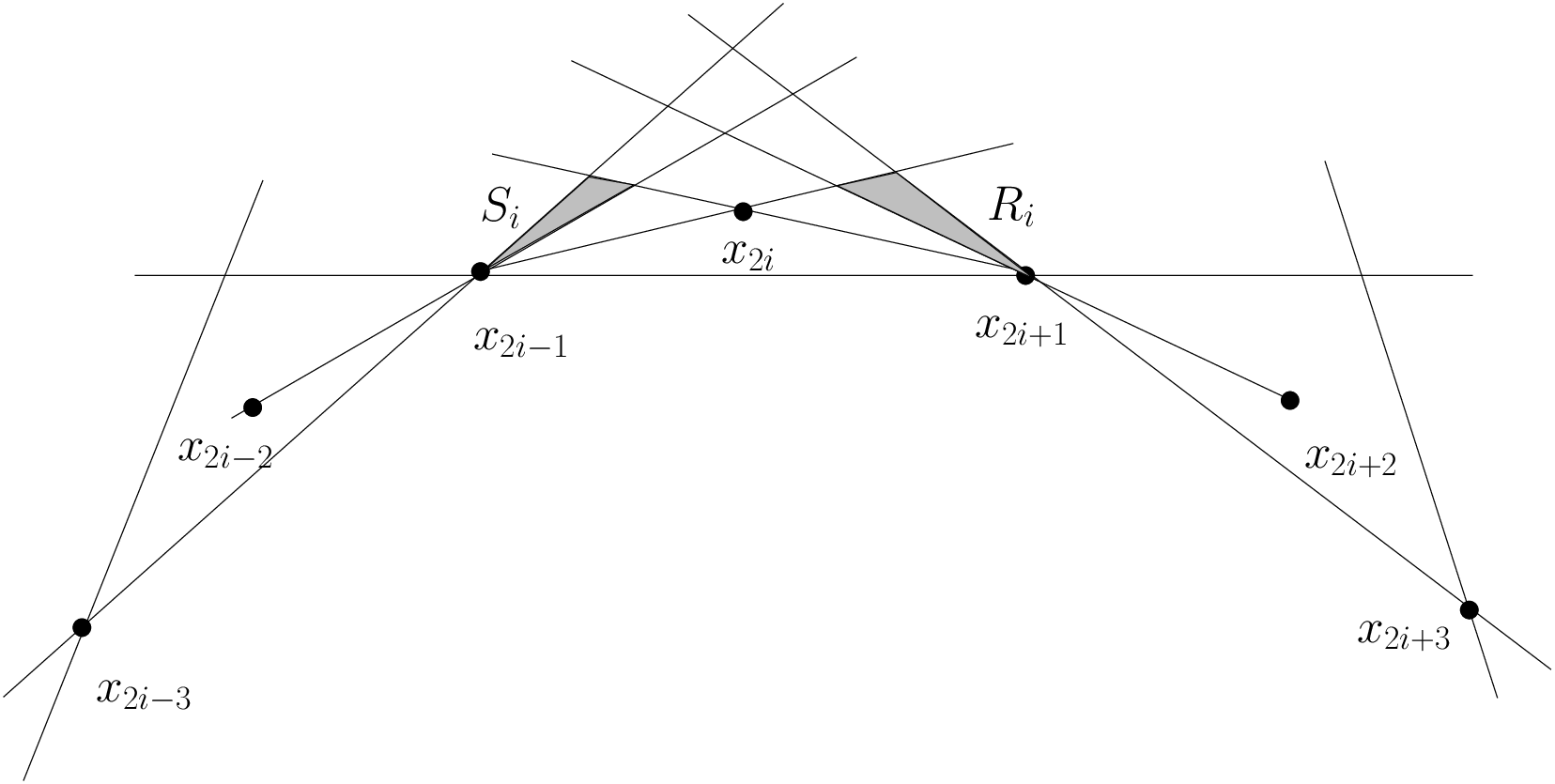}
        \caption{Regions $S_i$ and $R_i$ shaded in gray.}
        \label{fig:es1}
    \end{figure}

If $i - 1 \geq 2$, then again, we must have $x'_{2i - 2} \neq x_{2i - 2}$ since otherwise $$x'_{2i - 1} \in \conv(\{x'_{2i - 3},x'_{2i-2},x'_{2i}\})$$ and $X'$ would not be a $k$-hole.  By the same arguments as above, this implies that $x'_{2i - 2}$ must lie in $S_{i - 1}$ and we proceed to region $T_{i - 2}$.  Iterating this argument cyclically in counterclockwise order leads us to region $T_1$ with $x_2 \neq x'_2$ and $x'_2 \not\in R_1$ and $S_1  = \emptyset$.  Since both $X$ and $X'$ are $k$-holes, and neither uses a point in $T_{\lceil k/2\rceil}$, $X\cup \{x'_2\}$ is a $(k + 1)$-hole.  Contradiction.\end{proof}

This completes the proof.\end{proof}

\section{Concluding remarks}

For each fixed integer \(k\ge 6\), let \(h'_k(n)\) be the maximum number of
\(k\)-holes determined by an \(n\)-element point set in general position with
no \((k+1)\)-hole.  Thus \(h_k(n)=\max_{m\le n} h'_k(m)\). Let $h''_k(n)$ be the maximum number of $k$-holes determined by an $n$-element point set in general position with no $(k+1)$-hole or open $4$-cup. Thus, $h''_k(n) \leq h'_k(n)$.

Then the proof of Theorem~\ref{hole1} gives the same bounds for \(h'_k(n)\) and $h''_k(n)$.

\begin{theorem}
There are absolute constants $c_1,c_2>0$ such that the following holds. For
every integer \(n>k\ge 6\),
\[
\left(\frac{c_1}{k}\right)^{\lfloor k/3\rfloor} n^{\lfloor k/3\rfloor}
\le h''_k(n) \le h'_k(n)\le h_k(n) \le
\left(\frac{c_2}{k}\right)^{\lceil k/2\rceil} n^{\lceil k/2\rceil}.
\]
\end{theorem}

Unlike \(h_k(n)\), the functions \(h'_k(n),h''_k(n)\) are
not monotone by definition. We do not know whether monotonicity holds.

\begin{conjecture}
For all integers \(n>k\ge 6\), we have 
\[
h'_k(n)\le h'_k(n+1)
\qquad and
\qquad
h''_k(n)\le h''_k(n+1). 
\]
\end{conjecture}

Note that the following weaker monotonicity holds for $h''_k(n)$, which follows by placing a Horton set of size $N-n \geq n-1$ highly above the original point set, with the property that at least one point from the Horton set has its $x$-coordinate lying between the $x$-coordinates of any pair from the original point set. 

\begin{proposition}
For all integers $n>k\geq 6$, if $N\geq 2n-1$, then 
\[
h''_k(n)\le h''_k(N).
\]
\end{proposition}

\end{document}